\documentclass[12pt,twoside]{amsart}
\usepackage{amsmath}
\usepackage{amsthm}
\usepackage{amsfonts}
\usepackage{amssymb}
\usepackage{latexsym}
\usepackage[all]{xy}

\date{}
\allowdisplaybreaks[4] \footskip=15pt
\renewcommand{\uppercasenonmath}[1]{}

\numberwithin{equation}{section} \theoremstyle{plain}
\newtheorem*{thm*}{Main Theorem}
\newtheorem{thm}{Theorem}
\newtheorem{cor}[thm]{Corollary}
\newtheorem*{cor*}{Corollary}

\newtheorem*{lem*}{Lemma}

\newtheorem*{que*}{Question}
\newtheorem{prop}[thm]{Proposition}
\newtheorem*{prop*}{Proposition}
\newtheorem{rem}[thm]{Remark}
\newtheorem*{rem*}{Remark}
\newtheorem{exa}[thm]{Example}
\newtheorem*{exa*}{Example}
\newtheorem{df}[thm]{Definition}
\newtheorem*{df*}{Definition}

\newtheorem*{conj*}{Conjecture}

\newtheorem*{ack*}{ACKNOWLEDGEMENTS}

\newcommand{\pf}{\noindent\begin {proof}}
\newcommand{\epf}{\end{proof}}

\begin{document}
\begin{center}
{\large  \bf Two questions of L. Va${\rm \check{\textbf{s}}}$ on
$*$-clean rings}

\vspace{0.8cm} {\small \bf Jianlong Chen and Jian Cui\\
 \vspace{0.6cm} {\rm
Department of Mathematics, Southeast University
 \\ Nanjing 210096, P.R. China}\\}

 E-mail: {\it jlchen@seu.edu.cn, seujcui@126.com}

 \end{center}


\bigskip
\centerline { \bf  ABSTRACT}
 \bigskip
\leftskip10truemm \rightskip10truemm

\noindent A ring $R$ with an involution $*$ is called
(strongly) $*$-clean if every element of $R$ is the sum of a unit
and a projection (that commute). All $*$-clean rings are clean.
Va${\rm \check{s}}$ [L. Va${\rm \check{s}}$, $*$-Clean rings; some
clean and almost clean Baer $*$-rings and von Neumann algebras, J.
Algebra 324 (12) (2010) 3388-3400] asked whether there exists a
$*$-ring that is clean but not $*$-clean and whether a unit regular
and $*$-regular ring is strongly $*$-clean. In this paper, we answer
both questions by several examples. Moreover, some characterizations
of unit regular and $*$-regular rings are provided. \leftskip0truemm \rightskip0truemm
 \\

\bigskip
\section { \bf INTRODUCTION}
\bigskip

\indent Rings in which every element is the product of a unit and an
idempotent are said to be \emph{unit regular}, and have been
extensively studied. As a result due to Camillo and Khurana
\cite{Cam}, every element of a unit regular ring can also be written
as the sum of a unit and an idempotent.
 Recall that an element $a$ of a ring $R$ is
\emph{clean} if $a=e+u$ where $e^2=e\in R$ and $u$ is a unit of $R$,
and $R$ is called \emph{clean} if every element of $R$ is clean.
Clean rings were introduced by Nicholson \cite{Nicholson1977} in
relation to exchange rings. In 1999, Nicholson \cite{Nicholson99}
called an element of a ring $R$ \emph{strongly clean} if it is the
sum of a unit and an idempotent that commute with each other, and
$R$ is \emph{strongly clean} if each of its elements is strongly
clean. Clearly, a strongly clean ring is clean, and the converse
holds for abelian rings (that is, all idempotents in the ring are
central). Local rings and strongly $\pi$-regular rings are
well-known examples of strongly clean rings.

A ring $R$ is a \emph{$*$-ring} (or \emph{ring with involution}) if
there exists
an operation $*:R\rightarrow R$ such that for all $x,~y\in R$\\
$ \indent\indent\indent\indent\indent\indent(x+y)^*=x^*+y^*, \ \
(xy)^*=y^*x^*, \ \ ~{\rm and}
~(x^*)^*=x. $\\
 Clearly, $1^*=1$ and $0^*=0$ in any $*$-ring. An
element $p$ of a $*$-ring $R$ is said to be a \emph{projection} if
$p^2=p=p^*$. Recently, Va${\rm \check{s}}$ \cite{Va} introduced the
concepts of a $*$-clean ring and a strongly $*$-clean ring.
Following \cite{Va}, an element of a $*$-ring $R$ is called
(\emph{strongly}) \emph{$*$-clean} if it can be expressed as the sum
of a unit and a projection (that commute), and $R$ is called
\emph{$*$-clean} (\emph{resp., strongly $*$-clean}) in case all of
its elements are $*$-clean (resp., strongly $*$-clean). Strongly
$*$-clean rings are strongly clean and $*$-clean, and $*$-clean
rings are clean, but it is a question that whether there is an
example of a $*$-ring that is clean but not $*$-clean \cite{Va}.
According to \cite[Proposition 3, p. 229]{Ber72}, a $*$-ring $R$ is
\emph{$*$-regular} if one of the following equivalent conditions
hold: (1) $R$ is a (von Neumann) regular and Rickart $*$-ring (
i.e., the right annihilator of each element is generated by a
projection); (2) $R$ is regular and the involution is proper (that
is, $x^*x=0$ implies $x=0$ for any $x\in R$); $(3)$ for every $x$ in
$R$ there exists a projection $p$ such that $xR=pR.$ It was shown in
\cite{Va} that every $*$-abelian (i.e., a $*$-ring in which every
projection is central) and $*$-regular ring is strongly $*$-clean.
Va${\rm \check{s}}$ asked whether a unit regular and $*$-regular
ring is also strongly $*$-clean.

In this paper, examples of $*$-rings are provided to answer both
questions of Va${\rm \check{s}}$. Some properties of (strongly)
$*$-clean rings are investigated. In particular, we show that in
$*$-rings setting, a strongly clean ring is strongly $*$-clean iff
the set of all projections coincides with the set of all
idempotents. Several characterizations of unit regular and
$*$-regular rings are given.

Rings considered in this paper are associative with unity. The
notation $*$ denotes an involution of a given ring. The set of all
idempotents, all projections and all units of a ring $R$ are denoted
by $Id(R)$, $P(R)$ and $U(R)$, respectively. The symbol $l(X)$
(resp., $r(X)$) stands for the left (resp., right) annihilator of a
subset $X\subseteq R.$ We write $M_n(R)$ for the ring of all
$n\times n$ matrices over $R$.

\bigskip
\section { \bf  Main Results}
\bigskip
Let $R$ be a $*$-ring and $p\in P(R).$ The involution $*$ of $R$ is
inherited naturally to the corner ring $pRp$.

\begin{thm}\label{1}
Let $R$ be a $*$-ring and $p\in P(R).$ Then $a\in pRp$ is strongly
$*$-clean in $R$ if and only if $a$ is strongly $*$-clean in $pRp.$
\end{thm}

\begin{proof}
Assume that $a=e+u$ is strongly $*$-clean in $pRp$ with $e\in
P(pRp),$ $u\in U(pRp)$ and $ue=eu.$ Let $f=e+(1-p)$ and $v=u-(1-p)$.
Then $f\in P(R)$, $v\in U(R)$, and $f$ commutes with $v.$ So $a=f+v$
is strongly $*$-clean in $R.$

Conversely, suppose that $a\in pRp$ is strongly $*$-clean in $R$.
Let $a=e+u$ with $e\in P(R)$, $u\in U(R)$ and $ue=eu.$ Since
$a=pap,$ $1-p\in r(a)\cap l(a).$ By \cite[Theorem 2]{Nicholson99},
$r(a)\subseteq eR$ and $l(a)\subseteq Re.$ So we have $1-p\in eR\cap
Re=eRe,$ and then $(1-p)e=e(1-p),$ i.e., $ep=pe.$ This implies that
$pep\in Id(pRp).$ Note that both $e$ and $p$ are in $P(R).$ Thus
$pep\in P(pRp).$ Since $ap=pa$ and $ep=pe,$ $up=pu$. It follows that
$pup\in U(pRp),$ and $pep$ commutes with $pup$. Therefore,
$a=pep+pup$ is a strongly $*$-clean expression in $pRp.$
\end{proof}

\begin{cor}\label{2}
If $R$ is a strongly $*$-clean ring, then $pRp$ is strongly
$*$-clean for any $p\in P(R).$
\end{cor}

The following result, which reveals the relationship between strong
$*$-cleanness and strong cleanness, is crucial for constructing a
counter-example of a $*$-ring that is strongly clean but not
strongly $*$-clean.

\begin{thm}\label{4}
Let $R$ be a $*$-ring. Then $R$ is strongly $*$-clean if and only if
$R$ is strongly clean and $P(R)=Id(R).$
\end{thm}

\begin{proof}
Suppose that $R$ is strongly $*$-clean. The strong cleanness of $R$
is clear. For any $e^2=e\in R,$ we have $e=p+u$ where $p\in
P(R),~u\in U(R)$ and $e,~p$ and $u$ commute with each other. If
$p=0$ then $e=1,$ and if $p=1$ then $e=0.$ Notice that both $0$ and
$1$ are contained in $P(R).$ We may assume that $p\neq 0$ and $p\neq
1.$ Then $pRp$ and $(1-p)R(1-p)$ are nonzero $*$-rings. Now,
multiplying $e=p+u$ by $p$ yields $ep=p+up.$ It follows that
$-up=p-ep=(1-e)p\in U(pRp)\cap Id(pRp)=\{p\}$. Thus $ep=0.$
Analogously, by multiplying $1-p$ on both sides of $e=p+u$ we obtain
that $e(1-p)=u(1-p)\in U[(1-p)R(1-p)]\cap Id[(1-p)R(1-p)]=\{1-p\}$.
So one has $e(1-p)=1-p.$ Since $ep=0,$ $e=1-p.$ Clearly, $e=e^*.$
This proves that $Id(R)=P(R).$ The other direction is trivial.
\end{proof}

Due to \cite{Va}, if $R$ is a $*$-ring, the ring $M_n(R)$ has a
natural involution inherited from $R:$ if $A=(a_{ij})\in M_n(R),$
$A^*$ is the transpose of $(a_{ij}^*).$ Thus $M_n(R)$ is also a
$*$-ring. It was shown that $M_n(R)$ is a $*$-clean ring whenever
$R$ is $*$-clean \cite[Proposition 4]{Va}. By Theorem \ref{4}, we
have the following result.

\begin{cor}\label{3}
Let $R$ be a $*$-ring. Then $M_n(R)$ is not strongly $*$-clean for
$n\geq 2.$
\end{cor}

Note that a local ring $R$ with an involution $*$ is always strongly
$*$-clean. So $M_n(R)$ is $*$-clean, but it is not strongly
$*$-clean when $n\geq 2$. By \cite[Corollary 1.9]{Chen1}, there
exists a commutative local ring $R$ such that $M_2(R)$ is not
strongly clean. Va${\rm \check{s}}$ \cite{Va} asked whether there is
an example of a $*$-ring that is clean but not $*$-clean. We answer
this question affirmatively by the following example.

\begin{exa}\label{5}
Let $R=\mathbb{Z}_2\oplus \mathbb{Z}_2,$ where $\mathbb{Z}_2$ is the
ring of integers modulo $2$. It is clear that $R$ is a boolean ring.
Thus $R$ is strongly clean and $R=Id(R)$. Define a map $*:
R\rightarrow R$ by $(a,b)^*=(b,a)$. Since $R$ is commutative, it is
easy to check that $*$ is an involution of $R.$ Note that
$P(R)=\{(0,0),(1,1)\}\neq Id(R).$ In view of Theorem \ref{4}, $R$ is
not strongly $*$-clean, and thus not $*$-clean because $R$ is
commutative.
\end{exa}

\begin{rem}
Example \ref{5} shows that strongly clean $*$-rings need not be
$*$-clean. The following implications hold $($for the class of
$*$-rings$):$
$$\xymatrix
@R=2.5em @C=2em @M=0pt{strongly~*-clean~ring \ar@{=>}[d]\ar@{=>}[r] & ~*-clean~ring\ar@{=>}[d]\\
strongly~clean~ring\ar@{=>}[r]&~clean~ring }$$ In the table above,
each of the implications is irreversible, and there are no other
implications between these rings.
\end{rem}

Recall that a ring $R$ is \emph{right P-injective} if $lr(a)=Ra$ for
each $a\in R$. Regular rings are clearly right P-injective.

\begin{prop}\label{10}
Let $R$ be a $*$-ring. Then the following are equivalent$:$\\
$(1)$ $R$ is regular and the involution is proper $($i.e., $R$ is $*$-regular$)$.\\
$(2)$ $R$ is right P-injective and the involution is proper.\\
$(3)$ For every $a\in R,$ $Ra=Ra^*a.$
\end{prop}

\begin{proof}
$(1)\Rightarrow(2)$ is clear.

$(2)\Rightarrow(3).$ Given any $a\in R.$ Let $y\in r(a^*a).$ We have
$a^*ay=0,$ and it follows that $0=y^*a^*ay=(ay)^*(ay).$ Since the
involution is proper, $ay=0,$ i.e., $y\in r(a).$ Thus,
$r(a^*a)=r(a).$ By the right P-injectivity of $R$, we obtain
$Ra=lr(a)=lr(a^*a)=Ra^*a.$

$(3)\Rightarrow(1).$ For any $a\in R,$ there exists $t\in R$ such
that $a=ta^*a$. Then
$at^*a=(ta^*a)t^*a=t(a^*at^*)a=t(ta^*a)^*a=ta^*a=a.$ This proves
that $R$ is a regular ring. We finish by proving that the involution
is proper. Let $x^*x=0$ with $x\in R$. By $(3)$, $Rx=Rx^*x=0,$ so
$x=0,$ as desired.
\end{proof}

A ring $R$ is called \emph{strongly regular} if it is an abelian
regular ring, or equivalently, for any $a\in R,$ $a=eu=ue$ for $e\in
Id(R)$ and $u\in U(R)$ \cite{Nicholson99}.

\begin{prop}\label{11}
Let $R$ be a $*$-ring. Then the following are equivalent$:$\\
$(1)$ $R$ is strongly regular and the involution is proper.\\
$(2)$ $R$ is strongly regular and $P(R)=Id(R).$\\
$(3)$ $R$ is $*$-abelian and for every $a\in R$, $a=p+u$ with $aR
\cap pR=0$ where $p\in P(R)$ and $u\in U(R).$\\
$(4)$ For every $a\in R,$ $a=pu=up$ for some $p\in P(R)$ and $u\in
U(R).$
\end{prop}

\begin{proof}
$(1)\Rightarrow(2).$ In view of Proposition \ref{10}, $R$ is
$*$-regular. Since $R$ is also abelian, by \cite[Lemma 3]{Va}
$P(R)=Id(R)$.

$(2)\Rightarrow(3).$ Every abelian $*$-ring is $*$-abelian; and the
rest follows from \cite[Theorem 1]{Cam}.

$(3)\Rightarrow(4).$ Let $a\in R.$ Then there exists $1-p\in P(R)$
and $u\in U(R)$ such that $a=(1-p)+u$ and $aR\cap (1-p)R=0.$ Since
$R$ is $*$-abelian, $a(1-p)\in aR\cap R(1-p)=aR\cap (1-p)R=0.$
Hence, $a=ap=up=pu.$

$(4)\Rightarrow(1).$ The strong regularity of $R$ is clear. We
assume that $x^*x=0$ for $x\in R.$ Then $x=pu=up$ for some $p\in
P(R)$ and $u\in U(R).$ Obviously, $0=x^*x=u^*pu.$ Thus $p=0$, and so
$x=0.$ Therefore, the involution $*$ of $R$ is proper.
\end{proof}

A ring $R$ is said to \emph{have stable range $1$} provided that
whenever $aR+bR=R$ for any $a,~b\in R$, there exists $t\in R$ such
that $a+bt$ is a unit of $R.$ Next we give some characterizations of
unit regular and $*$-regular rings.
\begin{thm}\label{6}
Let $R$ be a $*$-ring. Then the following are equivalent$:$\\
$(1)$ $R$ is unit regular and the involution is proper.\\
$(2)$ $R$ is unit regular and $*$-regular.\\
$(3)$ For every $a\in R,$ $a=pu$ where $p\in P(R)$ and $u\in U(R)$.\\
$(4)$ For every $a\in R,$ $a=vq$ where $q\in P(R)$ and $v\in U(R)$.
\end{thm}

\begin{proof}
$(1)\Rightarrow(2)$ follows by Proposition \ref{10}.

$(2)\Rightarrow(3).$  For any $a\in R,$ there exists $e\in Id(R)$
and $w\in U(R)$ such that $a=ew$. Since $R$ is $*$-regular, $eR=pR$
for some projection $p\in R.$ Thus $e=pe.$ Note that $eR+(1-p)R=R.$
In view of \cite[Proposition 4.12]{Go}, $R$ has stable range $1$. So
there exists $t\in R$ satisfying $e+(1-p)t=v\in U(R)$. Clearly,
$pe=pv.$ It follows that $e=pe=pv$ and $a=ew=p(vw).$ Write $u=vw.$
Then $u\in U(R)$ and $a=pu.$

$(3)\Rightarrow(4).$ Given $a\in R,$ let $b=a^*.$ By hypothesis,
$b=pu$ with $p\in P(R)$ and $u\in U(R).$ Then $a=b^*=u^*p$. Taking
$v=u^*$ and $q=p,$ it follows that $v\in U(R),~q\in P(R)$ and
$a=vq.$

$(4)\Rightarrow(1).$ $R$ is clearly unit regular, so it suffices to
show that the involution is proper. Let $a\in R$ with $a^*a=0$. By
$(4),$ $a^*=vq$ for some $v\in U(R)$ and $q\in P(R).$ Thus
$0=a^*a=(vq)(qv^*)=vqv^*.$ So we have $q=0,$ which implies that
$a=0.$ We obtain the required result.
\end{proof}

\begin{df}
A $*$-ring $R$ is called $*$-unit regular if $R$ satisfies the
conditions in Theorem \ref{6}.
\end{df}

\begin{prop}\label{7}
Let $R$ be a $*$-ring and $n$ a positive integer. The following are
equivalent$:$\\
$(1)$ $M_n(R)$ is $*$-unit regular.\\
$(2)$ $R$ is unit regular and $a_1^*a_1+a_2^*a_2+\cdots +a_n^*a_n=0$
implies $a_i=0$ for all $i.$
\end{prop}

\begin{proof}
$(1)\Rightarrow(2).$ Since $M_n(R)$ is $*$-unit regular, it is unit
regular. By \cite[Corollary 4.7]{Go}, $R$ is unit regular. Suppose
that $a_1^*a_1+a_2^*a_2+\cdots +a_n^*a_n=0$ for some $a_i\in R.$ Let
$A=\left(
\begin{smallmatrix}
a_1 & 0 &\cdots& 0\\
a_{2}& 0& \cdots &0\\
\vdots &\vdots &\ddots & \vdots\\
 a_n& 0& \cdots &0
\end{smallmatrix} \right)\in M_n(R)$. Then $A^*A=0.$ Since the involution $*$ of $M_n(R)$ is proper, $A=0.$
Thus $a_1=a_2=\cdots=a_n=0.$

$(2)\Rightarrow(1)$. As $R$ is a unit regular ring, so is $M_n(R)$
by \cite[Corollary 4.7]{Go}. The remaining proof is to show that the
involution $*$ of $M_n(R)$ is proper. Let $A=(a_{ij})\in M_n(R)$
with $A^*A=0.$ We have $$ a_{1j}^*a_{1j}+a_{2j}^*a_{2j}+\cdots
+a_{nj}^*a_{nj}=0
$$ where $j=1,\ldots,n.$ Then, the
hypothesis implies that $a_{ij}=0$ for all $i,j.$ Thus we have
$A=0$, and the proof is complete.
\end{proof}

Based on Proposition \ref{7}, we have the following examples.
\begin{exa}\label{8} Let $\mathbb{R}$, $\mathbb{C}$ be the
fields of real
numbers and complex numbers, respectively. Clearly, both $\mathbb{R}$ and $\mathbb{C}$ are unit regular.\\
$(1)$ Let $*: \mathbb{R}\rightarrow \mathbb{R}$ be an involution
defined by $x\mapsto x$. Then $M_n(\mathbb{R})$ is $*$-unit regular.\\
$(2)$ Define an involution $*$ of the ring $\mathbb{C}$ by $x\mapsto
\bar{x},$ where $\bar{x}$ is the conjugation of $x.$ It can be
directly
checked that $M_n(\mathbb{C})$ is $*$-unit regular.\\
$(3)$ Let $R=\mathbb{R} \times \mathbb{R}$ be a ring with the usual
addition and multiplication. An involution $*$ of $R$ is given by $x\mapsto x.$ Then $R$ is unit regular and $M_n(R)$ is $*$-unit regular.\\
$(4)$ Let $*: x\mapsto x$ be an involution of $\mathbb{Z}_2$. Then
$M_2(\mathbb{Z}_2)$ is not $*$-unit regular because $1^*\cdot
1+1^*\cdot 1=0$ but $1\neq 0.$
\end{exa}

In \cite{Nicholson99}, Nicholson asked whether a unit regular ring
is strongly clean, it is still an open problem. Va${\rm \check{s}}$
\cite{Va} raised a question if a unit regular and $*$-regular ring
is strongly $*$-clean, which is equivalent to asking whether a
$*$-unit regular ring is strongly $*$-clean. Here we give a negative
answer.

\begin{exa}
Let $R$ be a $*$-ring as defined in Example \ref{8}$(1),(2)$ and
$(3)$. Then $M_2(R)$ is $*$-unit regular. Nevertheless, by Corollary
\ref{3} $M_2(R)$ is not strongly $*$-clean.
\end{exa}

According to Example \ref{8}(4), the matrix ring of a $*$-unit
regular ring need not be $*$-unit regular. However, we have the
following result for its corner ring.

\begin{prop}
If $R$ is a $*$-unit regular ring, then $pRp$ is $*$-unit regular
for every $p\in P(R).$
\end{prop}

\begin{proof}
The ring $R$ is unit regular, by \cite[Corollary 4.7]{Go} $pRp$ is
unit regular as well where $p\in P(R)$. Let $a\in pRp~ (\subseteq
R)$ with $a^*a=0.$ Since $R$ is $*$-unit regular, we deduce that
$a=0.$ So the involution in $pRp$ is proper. Thus $pRp$ is $*$-unit
regular by Theorem \ref{6}.
\end{proof}

\bigskip\bigskip

 \centerline {\bf ACKNOWLEDGMENTS} This research is supported
by the National Natural Science Foundation of China (10971024), the
Specialized Research Fund for the Doctoral Program of Higher
Education (200802860024), and the Natural Science Foundation of
Jiangsu Province (BK2010393).


\bigskip
\begin{thebibliography}{99}
\bibitem{Ber72} S.K. Berberian, Baer $*$-Rings, Grundlehren Math.
Wiss., vol. 195, Springer-Verlag, Berlin-Heidelberg-New York, 1972.

\bibitem{Cam} V.P. Camillo, D. Khurana, A characterization
of unit regular rings, Comm. Algebra 29 (5) (2001) 2293-2295.

\bibitem{Chen1} J.L. Chen, X.D. Yang, Y.Q. Zhou, On
strongly clean matrix and triangular matrix rings, Comm. Algebra 34
(10) (2006) 3659-3674.

\bibitem{Go} K.R. Goodearl, Von Neumann Regular Rings, Monogr.
Stud. Math., vol. 4, Pitman, London, 1979.

\bibitem{Nicholson1977} W.K. Nicholson, Lifting idempotents and exchange rings, Trans. Amer. Math. Soc. 229 (1977) 269-278.

\bibitem{Nicholson99} W.K. Nicholson, Strongly clean rings
and Fitting's lemma, Comm. Algebra 27 (8) (1999) 3583-3592.

\bibitem{Va} L. Va${\rm
\check{s}}$, $*$-Clean rings; some clean and almost clean Baer
$*$-rings and von Neumann algebras, J. Algebra 324 (12) (2010)
3388-3400.

\end{thebibliography}
\end{document}